\documentclass{article}
\usepackage{graphicx} 
\usepackage{amsmath}
\usepackage{amsfonts}
\usepackage{amssymb}
\usepackage{xcolor}
\usepackage{soul}
\usepackage{amsthm}
\usepackage{hyperref}
\usepackage{enumerate}
\usepackage[a4paper, total={6in, 8in}]{geometry}

\def\R{\mathbb{R}}

\def\H{\mathbb{H}}
\def\S{\mathbb{S}}
\def\V{\mathcal{V}}
\def\I{\mathcal{I}}
\def\st{\, | \,}
\def\dt{\dotsc}

{\theoremstyle{plain}
\newtheorem{thm}{Theorem}[section]
\newtheorem{lemma}[thm]{Lemma}
\newtheorem{cor}[thm]{Corollary}
\newtheorem{prop}[thm]{Proposition}

\newtheorem{defi}[thm]{Definition}
\newtheorem{innercustomthm}{Theorem}

}
\newenvironment{customthm}[1]
  {\renewcommand\theinnercustomthm{#1}\innercustomthm}
  {\endinnercustomthm}
  
{\theoremstyle{remark}
\newtheorem{rk}[thm]{Remark}

}

\newcommand{\fri}{{\mathrm{i}}}
\newcommand{\frj}{{\mathrm{j}}}
\newcommand{\frk}{{\mathrm{k}}}

\title{On the geometry of zero sets of central quaternionic polynomials II}
\author{Gil Alon, Adam Chapman and Elad Paran}

\begin{document}
\maketitle
\begin{abstract}
Following the work of the first and last authors \cite{AP24}, we further analyze the structure of a zero set of a left ideal in the ring of central polynomials over the quaternion algebra $\H$. We describe the ``algebraic hull" of a point in $\H^n$ and prove it is a product of spheres. Using this description we give a new proof to a conjecture of Gori, Sarfatti and Vlacci. We also show that the main result of \cite{AP24} does not extend to general division algebras.
\end{abstract}

\section{Introduction}

Let $R=\H[x_1,\dotsc,x_n]$ be the ring of polynomials in $n$ commuting variables over the Hamilton's ring of real quaternions, and consider the space of ``central" points
$$\H^n_c=\{ (q_1,\dotsc,q_n)\in \H^n \, | \, q_i q_j=q_jq_i \text{ for all } 1\leq i,j \leq n \}$$
For a polynomial $f=\sum_{i_1,\dotsc,i_n \geq 0} a_{i_1,\dotsc,i_n} {x_1}^{i_1}\dotsc {x_n}^{i_n}$ and any point $a=(a_1,\dotsc,a_n)\in \H^n$ let us define the substitution $f(a)$ by
$$f(a)=\sum_{i_1,\dotsc,i_n \geq 0} a_{i_1,\dotsc,i_n} {a_1}^{i_1}\dotsc {a_n}^{i_n}$$
For a subset $I\subseteq R$, let $\V(I)=\{a\in \H^n \, | \, f(a)=0 \, \text{ for all } f\in I \}$ and $\V_c(I)=\V(I) \cap \H^n_c$. For a subset $X\subseteq \H^n$, let $\I(X)=\{f\in R \, | \, f(a)=0 \, \text{for all } a\in X \}$.

In an earlier work of the first and last authors (\cite{AP24}), and also in an independent work of Gori, Sarfatti and Vlacci (\cite{GSV24}), the following theorem was proved:

\begin{thm}\label{thm:central_zeros}
Let $I$ be a left ideal of $R=\H[x_1,\dotsc,x_n]$. If a polynomial $f\in R$ vanishes on $\V_c(I)$, then $f$ vanishes on $\V(I)$.
\end{thm}

This theorem was used to prove the following Nullstellensatz, conjectured in \cite{GSV} by Gori, Sarfatti and Vlacci:

\begin{cor} \label{cor:quat_null}
Let $I$ be a left ideal of $R=\H[x_1,\dotsc,x_n]$. Then $\I(\V(I))=\sqrt{I}$.
\end{cor}

Here $\sqrt I$ is the intersection of {\it completely prime} left ideals containing $I$ (see \cite{AP21} and \cite{Reyes} for details).

The proof of Theorem \ref{thm:central_zeros} in \cite{AP24} was based on a geometric argument (which generalizes the argument for the case $n=2$ in \cite{GSV}): Given a point $v=(q_1,\dotsc,q_n)\in \V(I)$ which is not in $\H^n_c$, one constructs a different point $v' \in \V(I)$, obtained from $v$ by conjugating the coordinates $(q_{i+1},\dotsc,q_n)$ by a nonzero coordinate $q_i$ (where $1\leq i \leq n-1)$). One then finds a sphere $S \subseteq \H^n$ containing $v$ and $v'$, and proves that $S\subseteq \V(I)$. This process is then repeated by starting from points in $S$ and constructing from them new points of $\V(I)$, until reaching points $v_1, \dotsc, v_k$ in $\V_c(I)$. Under the assumption that $f$ vanishes on $\V_c(I)$, one concludes that $f(v_1)=\dotsc=f(v_k)=0$. These equalities and the specific construction of the points $v_1,\dotsc,v_k$ guarantee that $f(v)=0$.

The above proof scheme raises a natural question: Given a left ideal $I \subseteq R$ and a point $v\in \V(I)$, what is the natural set of points in $\H^n$ guaranteed to be in $\V(I)$ as well? More formally, we are looking for the following set, which we call the \emph{algebraic hull} of $v$:
\begin{align} \label{hull} \mathcal H(v)=\bigcap_{I \subseteq R \colon v\in \V(I)} \V(I),
\end{align}

where the intersection is over all left ideals $I$ in $R$ with $v \in \V(I)$.

To answer this question, let us consider a minimal decomposition of $v$ to central sub-vectors: $v=(v_0,v_1,\dotsc,v_r)$ where $v_i\in \H^{d_i}_c$ and $\sum d_i=n$. We consider the set 

\begin{align} \label{eq:B(v)}
B(v)=\{v_0\} \times \S_{v_1} \times \dotsc \times \S_{v_r}
\end{align}
where $\S_w=\{qwq^{-1} \, | \, q\in \H^* \}$ is the sphere defined by $w$. We shall call this set the \emph{enveloping multisphere} of $v$. 
\footnote{Note that the minimal decomposition $v=(v_0,\dt,v_r)$ to central sub-vectors is not unique, as some of the coordinates may be in $\R$, but changing the subdivision of $v$ by moving a real coordinate from one sub-vector to another does not change the product (\ref{eq:B(v)}).}
 We will prove that $\mathcal H(v)=B(v)$. The proof has two parts: First, we will show the inclusion $B(v)\subseteq \mathcal H(v)$ by proving the following theorem:

\begin{thm} \label{thm:envelope}
Let $I$ be a left ideal in $R=\H[x_1,\dotsc,x_n]$. If $v\in \V(I)$ then $B(v) \subseteq \V(I)$.
\end{thm}

Then, we will prove the inclusion $\mathcal H(v)\subseteq B(v)$ by showing that $B(v)$ itself is algebraic, i.e: $B(v)=I(J)$ for some left ideal $J$ in $R$.

Furthermore, we will investigate the sets of the form $$\S(v_0;v_1,\dotsc,v_r)=\{v_0\} \times \S_{v_1} \times \dotsc \times \S_{v_r}$$ for any vectors $v_0 \in \H^{k_0}$ and $v_i \in \H^{k_i}_c \setminus \R^{k_i}$. We shall call such sets \emph{multispheres}. We shall prove the following properties, which may be of independent interest:

\begin{enumerate}
    \item The restriction of a slice regular polynomial $f\in R$ to a multisphere is a multi-affine function (Lemma \ref{lem:affine_on_sphere}).
    \item Given a multisphere $\S(v_0;v_1,\dotsc,v_r)$ and two-point subsets $W_1 \subset \S_{v_1},\dotsc,$ $W_r \subset \S_{v_r}$, if a polynomial $f\in R$ vanishes on $ \{v_0\} \times W_1 \times \dotsc \times W_r $, then $f$ vanishes on the entire multisphere (Lemma \ref{lem:vanish_on_sphere}).
\end{enumerate}

Using these properties and Theorem \ref{thm:envelope}, we will present a new proof of Theorem \ref{thm:central_zeros}. 

Finally, we will answer the question presented at the end of \cite{AP24}, whether Theorem \ref{thm:central_zeros} holds if one replaces the quaternion algebra $\H$ with an arbitrary division ring. Our answer is negative. We will construct a quaternion algebra over a certain field for which the claim of the theorem does not hold. Our proof will use techniques from the theory of quadratic forms over fields.

The rest of this paper is organized as follows: In section \ref{sec:prelim} we quickly recall a few basic facts about the quaternion algebra $\H$, and introduce some convenient notation. In section \ref{sec:spheres} we introduce the notion of a multisphere (which generalizes the notion of an embedded sphere from \cite{AP24}), and prove some useful properties. We also define the enveloping multisphere of a point. In section \ref{sec:hull} we show that indeed the algebraic hull of a point is equal to its enveloping multisphere. In section \ref{sec:proof} we give a new proof to Theorem \ref{thm:central_zeros}, based on these new concepts. Finally, in section \ref{sec:counterexample} we show that Theorem \ref{thm:central_zeros} does not hold over arbitrary division rings.

\section{Preliminaries} \label{sec:prelim}

Let us recall some definitions and results from \cite{AP24}. We let $\H=\R + \R \fri + \R \frj + \R \frk$ be Hamilton's algebra of quaternions,\footnote{We denote by $\fri, \frj,\frk$ the standard generators of $\H$, as opposed to the letters $i,j,k$ which we reserve for indices.} and set $\S = \{q \in \H \, | \, q^2=-1 \}$. For any $s\in \S$, let $L_s=\R + \R s$ (the maximal subfield of $\H$ determined by $s$). 
For any $q\in \H^*$ and $v=(q_1,\dotsc,q_n)\in \H^n$, let us denote $v^q=(qq_1q^{-1},\dotsc,q q_n q^{-1}).$ Let us also denote for such $v$, 
$\S_v=v^{\H^*}$.

We recall from \cite{AP24} the following properties of $\H$ and $\S$:

\begin{lemma}\label{lem:quaternion_facts} We have:
\begin{enumerate}
\item
$\S=\{a \fri+b \frj+c \frk\, | \,a,b,c\in \R, a^2+b^2+c^2=1\}$
\item 
Any $q\in \mathbb H$ has a representation as $q=a+bs$ where $a,b\in \mathbb {R}$ and $s\in \S$. 
\item
Any two elements of $\S$ are conjugate in the multiplicative group $\H^*$.
\item
We have $\H^n_c=\bigcup_{s\in \S}(L_s)^n$.
\item
For any $v\in \H^n_c$, $v$ can be expressed as $v=A+Bs$ where $A,B\in \R^n$ and $s\in \S$, and we have $\S_v=\{ A+ Bs \st s \in \S \} = A+B \S$.
\end{enumerate}
\end{lemma}

\begin{proof}
    See \cite[Lemma 2.3 and equations (1),(2)]{AP24}.
\end{proof}

\section{Multispheres} \label{sec:spheres}
\begin{defi}\label{def:multisphere}
Let $n=\sum_{i=0}^rk_i$ where $k_0\geq 0$ and $k_i\geq 1$ for $i\geq 1$. Let $v_0\in \H^{k_0}$ and $v_i \in \H^{k_i}_c \setminus \R^{k_i}$ for $1\leq i \leq r$. The set $$ \S(v_0; v_1,...,v_r)=\{v_0\} \times \prod_{i=1}^r \S_{v_i} \subseteq \H^n$$ is called an $r$-multisphere in $\H^n$.
\end{defi}

We note that this notion generalizes the notion of an embedded sphere from \cite[Definition 3.2]{AP24}: An embedded sphere is a $1$-multisphere.

By part 5 of Lemma \ref{lem:quaternion_facts}, if we let for each $1\leq i\leq r$, $v_i=A_i+B_i t_i$ with $A_i\in \R^{k_i}, B_i \in \R^{k_i}\setminus \{0\}$, $t_i \in \S$, then 
\begin{align} \label{eq:multisphere_coords}
\S(v_0; v_1,...,v_r)=\{(v_0, A_1+B_1 s_1,\dotsc, A_r+B_r s_r) \st s_1,\dotsc, s_r \in \S \}
\end{align}

Therefore, by part 1 of Lemma \ref{lem:quaternion_facts}, an $r$-multisphere is topologically equivalent to a product of $r$ $2$-spheres.

\begin{defi}
A polynomial $p\in \H[x_1,\dotsc,x_k]$ is called multi-affine if in each of its monomials, all the variables have a degree of $0$ or $1$.
\end{defi}

The following lemma shows that on multispheres, slice regular polynomials are multi-affine functions.

\begin{lemma}\label{lem:affine_on_sphere}
Let $p\in \H[x_1,\dotsc, x_n]$ and suppose that $n=\sum_{i=0}^rk_i$ where $k_0\geq 0$ and $k_i\geq 1$ for $i\geq 1$. Let $v_0\in \H^{k_0}$ and $A_i, B_i\in \R^{k_i}$ for $1 \leq i \leq r$. Then there exists a multi-affine polynomial $q\in \H[y_1,\dotsc, y_r]$ such that for all $s_1,\dotsc,s_r\in \S$ we have 
\begin{align} \label{eq:multiaffine_rep_general} p(v_0, A_1+B_1 s_1,\dotsc, A_r+B_r s_r)=q(s_1,\dotsc,s_r)
\end{align}
   
\end{lemma}
\begin{proof}
Let us relabel the variables $x_1,\dotsc,x_n$ according to the partition $n=\sum_{i=0}^rk_i$:
\begin{align} \label{eq:rename_variables}
(x_1,\dotsc,x_n)=(x_{0,1},\dotsc,x_{0,k_0},\, x_{1,1},\dotsc,x_{1,k_1}, \, \dotsc,\,x_{r,1},\dotsc,x_{r,k_r}) 
\end{align}
\begin{enumerate}[\emph{Step }1.]
\item
For each $1\leq j\leq r$, consider a monomial 
\begin{align} \label{eq:monomial} 
m_j(x_{j,1},\dotsc,x_{j,k_j})=x_{j,1}^{i_1}\dotsc x_{j,k_j}^{i_{k_j}}
\end{align}
We claim that there exist $c_j,d_j \in \R$  such that for all $s \in \S$, 

$$m_j(A_j+B_j s)=c_j + d_j s$$
This claim follows by induction on $\deg m_j$, and noting that for $a,b,c,d\in \R$ and $s\in \S$, $$(a+bs)(c+ds)=(ac-bd)+(ad+bc)s$$

\item Let us consider a monomial $m=m_1 m_2 \dotsc m_r$ where each $m_i$ is as in (\ref{eq:monomial}). Then for all $s_1,\dotsc,s_r\in S$ we have $$m(A_1+B_1s_1,\dotsc,A_r+B_rs_r)=(c_1+d_1s_1)(c_2+d_2s_2)\dotsc(c_r+d_rs_r)$$
Hence, there exists a multi-affine polynomial $q \in \R[y_1,\dotsc,y_r]$ such that 
\begin{align} \label{eq:multiaffine_rep_monomial}
m(A_1+B_1s_1,\dotsc,A_r+B_rs_r)=q(s_1,\dotsc,s_r)
\end{align}
\item Let $p\in \H[x_1,\dotsc,x_n]$ and let us write $p=\sum p_i m_i$ where 

$p_i \in \H[x_{0,1},\dotsc,x_{0,k_0}]$ and $m_i$ are monomials as in (\ref{eq:monomial}). By the previous step, there exist multi-affine polynomials $q_i$ which satisfy (\ref{eq:multiaffine_rep_monomial}) with respect to $m_i$. Let us define $q=\sum_i p_i(v_0) q_i$. Then $q\in \H[y_1,\dotsc,y_r]$ is a multi-affine polynomial and
\begin{align*}
    p(v_0, A_1+B_1 s_1,\dotsc, A_r+B_r s_r) &=\sum p_i(v_0) m_i(A_1+B_1 s_1,\dotsc, A_r+B_r s_r) \\
    &= \sum p_i(v_0) q_i(s_1,\dotsc,s_r)  \\
    &= q(s_1,\dotsc,s_r) \qedhere
\end{align*}

\end{enumerate}
\end{proof}

From Lemma \ref{lem:affine_on_sphere} we conclude:
\begin{lemma} \label{lem:vanish_on_sphere} 
Let $S=\S(v_0; v_1,\dots,v_r)$ be an $r$-multisphere. For each $1\leq i \leq r$, let $w_{i,1}$ and $w_{i,2}$ be two distinct points in $\S_{v_i}$. Consider the set 
$$ Q=\{v_0\} \times \{w_{1,1},w_{1,2}\} \times \dotsc \times \{w_{r,1},w_{r,2}\} \subseteq S $$

If a slice regular polynomial $p\in \H[x_1,\dotsc,x_n]$ vanishes on $Q$, then $p$ vanishes on the entire multisphere $S$.

\end{lemma}

\begin{proof} Let $v_i= A_i+B_i t_i$ with $A_i, B_i \in \R^{k_i}$ and $t_i \in \S$. By Lemma \ref{lem:affine_on_sphere} there exists a multi-affine polynomial $q\in \H[y_1,\dotsc,y_r]$ satisfying (\ref{eq:multiaffine_rep_general}) for all $s_1,\dotsc,s_r\in \S$. By Lemma 
\ref{lem:quaternion_facts}, there exists for each $1\leq i \leq r$, $s_{i,1},s_{i,2} \in \S$ such that $w_{i,j}=A_i+B_i s_{i,j}$ for $j=1,2$. Therefore $s_{i,1} \neq s_{i,2}$. By (\ref{eq:multiaffine_rep_general}), $q$ vanishes on the set $$\prod_{i=1}^r \{s_{i,1},s_{i,2}\}$$
We claim that a multi-affine polynomial $q\in \H[y_1,\dotsc,y_r]$ which vanishes on such a product is the zero polynomial. Let us show it by induction. For $r=1$, $q$ is of the form $q(y_1)=ay_1+b$, and by plugging in $y_1=s_{1,1}$ and $y_1=s_{1,2}$ and solving the linear equations, we get $q=0$. For $r>1$, we have $q=q_0+q_1 y_r$ where  $q_0,q_1 \in \H[y_1,\dotsc,y_{r-1}]$ are multi-affine polynomials. Hence, for any 
$a_1,\dotsc,a_r \in \H^r$ we have $q(a_1,\dotsc,a_r)=q_0(a_1,\dotsc,a_{r-1})+q_1(a_1,\dotsc,a_{r-1})a_r$. For any $(a_1,\dotsc,a_{r-1}) \in 
\prod_{i=1}^{r-1} \{s_{i,1},s_{i,2}\}$, let $a=q_1(a_1,\dotsc,a_{r-1})$ and $b=q_0(a_1,\dt,a_{r-1})$.
Then both $s_{r,1}$ and $s_{r,2}$ satisfy the equation $ay+b=0$, so $a=b=0$. By the induction hypothesis we get that $q_0=q_1=0$, so $q=0$, and the induction is complete.

Therefore, $q=0$ and by (\ref{eq:multiaffine_rep_general}), $p$ vanishes on $S$.
\end{proof}

\begin{rk}
Lemmas \ref{lem:affine_on_sphere} and \ref{lem:vanish_on_sphere} generalize Lemmas 3.1 and 3.3 in \cite{AP24}, which in turn generalize some properties of slice regular functions in one variable \cite[Section 1.2 and Lemma 3.1]{GSS}. 
\end{rk}

\begin{defi} \label{def:envelope}
\begin{enumerate}
    \item Given a vector  $v\in  \H^n$ let us write $v=(v_0,\dotsc,v_r)$ where $v_i \in \H^{k_i}_c$, $v_i=(v_{i,1}, \dotsc, v_{i,k_i})$ such that $(v_{i,k_i}, v_{i+1}) \notin \H^{k_{i+1}+1}_c$ for all $0\leq i \leq r-1$ (note that this presentation is unique). We call this presentation the central presentation of $v$.
    \item For any vector $v \in \H^n$, let $v=(v_0,\dotsc,v_r)$ be its central presentation. Denote
     $$B(v)=\S(v_0; v_1,\dotsc, v_r)$$ We call $B(v)$ the {\it enveloping multisphere} of $v$.
\end{enumerate}
\end{defi}

The following properties of the central presentation are easy to see:
\begin{lemma} Let $v\in \H^n$, and let $v=(v_0,\dotsc,v_r)$ be its central presentation.
    \begin{enumerate}    
        \item $r=0$ if and only if $v\in \H^n_c$.
        \item for $1\leq i \leq r$, $v_i \notin \R^{k_i}$.
    \end{enumerate}
\end{lemma}

\section{The algebraic hull} \label{sec:hull}
Recall the definition of the algebraic hull $\mathcal H(v)$ of a point $v\in \H^n$ as the intersection of all the algebraic sets in $\H^n$ containing $v$ (see (\ref{hull}) in the introduction). In this section we will prove that $\mathcal H(v)=B(v)$. We start by recalling the following lemma:
\begin{lemma}[{\cite[Lemma 4.1]{AP24}}] \label{lem:conj_root}
Let $I$ be a left ideal in $\H[x_1,\dotsc,x_n]$ and let $v=(q_1,\dotsc,q_n)\in \V(I)$. Let $1\leq i\leq n-1$ such that $q_i\neq 0$, and define $v_1=(q_1,\dotsc,q_i)$ and $v_2=(q_{i+1},\dotsc,q_n)$. Then $(v_1,v_2^{q_i})\in \V(I)$.    
\end{lemma}

\begin{customthm}{\ref{thm:envelope}}
Let $I$ be a left ideal in $R=\H[x_1,\dotsc,x_n]$. If $v\in \V(I)$ then $B(v) \subseteq \V(I)$.
\end{customthm}

\begin{proof}
Let $v=(v_0,\dotsc,v_r)$ be the central representation of $v$. We will prove by a descending induction on $i$ that for all $0\leq i \leq r$, 
\begin{align} \label{eq:inductive_claim}
\{(v_0, \dotsc,v_i)\} \times \S_{v_{i+1}} \times \dotsc \times \S_{v_r} \subseteq  \V(I)
\end{align}
    For  $i=r$ we  have $\{(v_0,\dotsc,v_r)\}=\{v\} \subseteq \V(I)$. Let us  assume that the claim is true for some $i>0$.    Let $w=v_{i-1,k_{i-1}}$. By lemma \ref{lem:conj_root}, 
    $$\{(v_0, \dotsc,v_{i-1})\} \times \left (\{v_i  \} \times \S_{v_{i+1}} \times \dotsc \times \S_{v_r} \right )^w\subseteq  \V(I)$$ 
 
    By Lemma \ref{lem:quaternion_facts} we have for all $j$, ${\S_{v_j}}^w=\S_{v_j}$. Hence
    $$\{(v_0, \dotsc,v_{i-1},{v_i}^w) \} \times \S_{v_{i+1}} \times \dotsc \times \S_{v_r} \subseteq  \V(I)$$ 
    For each $i+1\leq j \leq r$ let $W_j\subseteq \S_{v_j}$ be a set of size 2 (note that $\S_{v_j}$ is not a singleton because $v_j \notin \R^{k_j}$). Let $W_i=\{v_i, {v_i}^w\}$. Then we have 
    $$\{(v_0, \dotsc,v_{i-1})\} \times W_i \times \dotsc \times W_r  \subseteq  \V(I)$$ 
    By Lemma \ref{lem:vanish_on_sphere}, 
    $$\{(v_0, \dotsc,v_{i-1})\} \times \S_{v_i} \times \dotsc \times \S_{v_r} \subseteq  \V(I)$$
    Our induction is therefore complete. Letting $i=0$ in (\ref{eq:inductive_claim}) we get $B(v)\subseteq \V(I)$.
\end{proof}

The above theorem clearly implies that $B(v) \subseteq \mathcal H(v)$. To prove that $\mathcal H(v) \subseteq B(v)$, we shall find a left ideal $J$ in $R$ such that $B(v)=\V(J)$. We start with the following lemma:

\begin{lemma} \label{lem:multiple_of_real}
Let $1\leq n_1\leq n_2 \leq n$ and $p \in \R[x_{n_1},\dotsc, x_{n_2}]$. Let $v
=(q_1,\dotsc,q_n)\in \H^n$ be a point satisfying $(q_{n_1},\dotsc,q_{n_2})\in \H^{n_2-n_1+1}_c$ and $p(v)=0$. Then $v \in \V(Rp) $ where $R=\H[x_1,\dotsc,x_n]$.
\end{lemma}

\begin{proof}
Let us consider a monomial $m=am_1m_2m_3$ where $a\in \H$, $m_1=\prod_{j=1}^{n_1-1}{x_j}^{i_j}$, $m_2=\prod_{j=n_1}^{n_2}{x_j}^{i_j}$ and $m_3=\prod_{j=n_2+1}^{n}{x_j}^{i_j}$. Let $v_1=(q_1,\dotsc,q_{n_1-1})$, $v_2=(q_{n_1},
\dotsc,q_{n_2})$ and $v_3=(q_{n_2+1},\dotsc, q_n)$. Since $v_2 \in \H^{n_2-n_1+1}_c$, we have 
$(m_2p)(v_2)=m_2(v_2)p(v_2)=0$. Since $p$ is in the center of $R$, we have
$$(mp)(v)=(am_1m_2pm_3)(v)=a m_1(v_1) (m_2p)(v_2) m_3(v_3)=0 $$ As $R$ is additively spanned by monomials of the form of $m$, our proof is complete.
\end{proof}

\begin{thm} \label{thm:multisphere_is_algebraic}
    Let $v\in \H^n$. There exists a left ideal $J \subseteq R$ such that $B(v)=\V(J)$.
\end{thm}

\begin{proof}
Let $v=(v_0,\dotsc,v_r)$ be the central presentation of $v$, where $v_i\in \H^{k_i}_c$. By Lemma \ref{lem:quaternion_facts}, for each $i$ there exist $A_i,B_i \in \R^{k_i}$ and $s_i \in \S$ such that $v_i= A_i+B_i s_i$, and we have

$$ B(v)= \{v_0\} \times (A_1+B_1 \S) \times \dotsc \times (A_r+B_r \S) .$$

Let us rename the variables $x_1,\dotsc,x_n$ as in (\ref{eq:rename_variables}).
Let $v_0=(q_1,\dotsc,q_{k_0})$.
Consider the following equations:
\begin{align}\label{eq:msphere_eqs}
\left \{ 
\begin{array}{ll}
(x_{i,j}-A_{i,j})B_{i,l}=(x_{i,l}-A_{i,l})B_{i,j} & \text{for } 1\leq i \leq r, \, 1\leq j,l \leq k_i \\
(x_{i,j}-A_{i,j})^2=-{B_{i,j}}^2 & \text{for } 1\leq i \leq r, \, 1\leq j \leq k_i \\
x_{0,j}=q_j & \text{for } 1\leq j \leq k_0
\end{array}
\right.
\end{align}
We will now show that these equations and their left multiples cut out the multisphere $B(v)$.

Consider the set $$G=\left \{(x_{i,j}-A_{i,j})B_{i,l}-(x_{i,l}-A_{i,l})B_{i,j},  (x_{ij}-A_{ij})^2+{B_{i,j}}^2, x_{0,j}-q_j \right \}$$
where $i,j,l$ are as in (\ref{eq:msphere_eqs}), and let $J$ be the left ideal generated by $G$. 

We will now prove that $B(v)=\V(J)$: Let $w\in \V(J)$. Then in particular, $w$ is in the zero set of $G$, so equations (\ref{eq:msphere_eqs}) hold. 
Let $w=(w_{i,j})$ with $0\leq i \leq r$ and $1\leq j \leq k_i$. We have $w_{0,j}=q_j$ for all $j$. Let $1\leq i \leq r$ and $1\leq j \leq k_i$, then by the second equation in (\ref{eq:msphere_eqs}), there exists $s_{i,j} \in \S$ such that $w_{i,j}=A_{i,j}+B_{i,j}s_{i,j}$. Plugging this equality in the first equation of (\ref{eq:msphere_eqs}), we get $ B_{i,j}B_{i,l}s_{i,j}= B_{i,l}B_{i,j}s_{i,l}$, so unless $B_{i,j}=0$ or $B_{i,l}=0$, we have $s_{i,j}=s_{i,l}$. We conclude that $(w_{i,1},\dotsc,w_{i,k_i}) \in A_i + B_i \S$. Hence, $w\in B(v)$.

For the converse direction, let us assume that $w\in B(v)$. Then we have 
$$w = (q_1,\dotsc, q_{k_0}, A_1 + B_1 s_1, \dotsc, A_r + B_r s_r)$$
for some $s_1,\dotsc, s_r \in \S$. Equations (\ref{eq:msphere_eqs}) hold for such $w$, so $p(v)=0$ for all $p\in G$. Let $g\in G$ and $p\in \H[x_1,\dt,x_n]$. If $g$ is a generator of the first or second kind, then by Lemma \ref{lem:multiple_of_real} we have $(pg)(w)=0$. If $g$ is a generator of the third kind, then $g=x_{0,j}-q_{0,j}$ for some $j$. By \cite[Proposition 3.6]{GSV}, $pg$ vanishes on $(C_{q_{0,j}})^{j-1}\times \{ q_{0,j} \} \times \H^{n-j}$, where $C_a=\{q\in \H \st qa=aq \}$. In particular, $(pg)(w)=0$. 
We have proved that all the left multiples of elements of $G$ vanish on $w$, therefore $w\in \V(J)$.
 \end{proof}

From Theorems \ref{thm:envelope} and \ref{thm:multisphere_is_algebraic} we conclude:
\begin{thm}
    For any $v\in \H^n$ we have $\mathcal H(v)=B(v).$
\end{thm}

\section{A new proof of Theorem \ref{thm:central_zeros}} \label{sec:proof}
Using the results in the preceding section, we now present a new proof of Theorem \ref{thm:central_zeros}. The idea of the proof is that the enveloping multisphere allows us to obtain from a point $v\in \V(I)$ some points in $\V_c(I)$ in one step, instead of many steps of construction of $1$-multispheres as in the previous proof in \cite{AP24}.

\begin{customthm}{\ref{thm:central_zeros}}
Let $I$ be a left ideal of $R=\H[x_1,\dotsc,x_n]$. If a polynomial $f\in R$ vanishes on $\V_c(I)$, then $f$ vanishes on $\V(I)$.
\end{customthm}

\begin{proof}
Let us assume that $f\in R$ vanishes on $\V_c(I)$, and let $v \in \V(I)$. If $v\in \V_c(I)$ then clearly, $f(v)=0$. Otherwise, let $v=(v_0,\dotsc,v_r)$ be the central presentation of $v$. By Theorem \ref{thm:envelope}, $B(v)\subseteq \V(I)$. For each $1 \leq i \leq r$, let $v_i=A_i + B_i s_i$  with $A_i, B_i \in \R^{k_i}$ and $s_i \in \S$. Since $v_i \notin \R^{k_i}$, we have $B_i \neq 0$. Consider the following set of $2^r$ points:
$$Q=\{(v_0, A_1\pm B_1 s_0, A_2 \pm B_2 s_0, \dotsc A_r \pm B_r s_0)\}$$
For all $1\leq i \leq r$ we have by Lemma \ref{lem:quaternion_facts}, $A_i\pm B_i s_0 \in A_i + B_i \S = \S_{v_i}$, so $Q \subseteq B(v)$, hence $Q \subseteq \V(I)$. On the other hand, $Q \subseteq (L_{s_0})^n \subseteq \H^n_c$. We conclude that $Q \subseteq \V_c(I)$. Therefore, $f$ vanishes on $Q$. By Lemma \ref{lem:vanish_on_sphere}, $f$ vanishes on $B(v)$. As $v\in B(v)$, we conclude that $f(v)=0$. \qedhere
\end{proof}

\section{A counterexample} \label{sec:counterexample}

Let $D$ be any division algebra. As in the case $D=\H$, we may consider the central polynomial ring $R_D=D[x_1,\dotsc,x_n]$ and the spaces $D^n$ and $D^n_c=\{ (q_1,\dotsc,q_n)\in D^n \, | \, q_i q_j=q_jq_i \text{ for all } 1\leq i,j \leq n \}$.
We may also define the vanishing sets $\V(I), \V_c(I)$ of an ideal $I\subseteq R_D$, as in the quaternionic case.

In \cite{AP24} the following question was asked: Given a left ideal  $I \subseteq R_D$ and a polynomial $f\in R_D$ which vanishes on $\V_c(I)$, does $f$ necessarily vanish on $\V(I)$?

In this section we shall prove that the answer to this question is negative, by constructing a quaternion algebra $Q$ and an ideal $I$ in $R_Q$ such that $\V_c(I)$ is empty, but $\V(I)$ is not. Therefore, any polynomial $f$ which does not vanish on $\V(I)$ will provide a counterexample.

In order to understand the argument of this construction, one must be familiar with the algebraic theory of quadratic forms over fields. We refer the reader to \cite{LAM} for a thorough introduction.

Recall that given a field $F$ of characteristic not 2, a quaternion algebra $Q$ takes the form of a Hilbert symbol, which means $Q=(\alpha,\beta)_F=F\langle i,j | i^2=\alpha, j^2=\beta, ji=-ij \rangle$ for some $\alpha,\beta \in F^*$.

\begin{prop}\label{counter_ex}
    Let $k$ be a field of characteristic not 2 with no square root of $-1$ (such as $\mathbb{R}$). Set $F=k(\alpha,\beta,t)$ to be the function field in three algebraically independent variables over $k$, and $Q=(\alpha,\beta)_F$. Then the polynomial $p=x^2-\alpha+t(y^2-\beta)\in Q[x,y]$ has no roots in $Q_c^2$.
    
\end{prop}

\begin{proof}
    The claim is that $p$ has no root $(r,s)$ where $r$ and $s$ belong to a subfield $K$ of $(\alpha,\beta)_F$.
    Assume the contrary.
    Note that $p=x^2+ty^2+(-\alpha-\beta t)$ and thus $p$ is obtained from $\pi=x^2+ty^2+(-\alpha-\beta t)z^2$ by plugging in $z=1$. The latter is written (using standard notation for quadratic forms) as $\pi=\langle 1,t,-(\alpha+\beta t) \rangle$. The existence of a root in $K$ for $p$ implies that $\pi$ is isotropic over $K$ (which means it has a root different from the zero vector). Now, $\pi$ is a subform of $\tau=\langle 1,t,-(\alpha+\beta t),t(\alpha+\beta t) \rangle$, commonly denoted by $\langle \! \langle -t,\alpha+\beta t \rangle \! \rangle$ in the language of Pfister forms. Since $\pi$ is isotropic over $K$, so is the latter. The form $\tau$ is the norm form of the quaternion algebra $A=(-t,\alpha+\beta t)_F$, which means that $A$ is split by $K$. Therefore, the biquaternion algebra $A \otimes_F Q$ is not a division algebra, and so the underlying Albert form $\varphi=\langle -t,\alpha+\beta t,t(\alpha+\beta t),-\alpha,-\beta,\alpha \beta \rangle$ is isotropic.
    Consider the $(\alpha+\beta t)$-adic valuation. The two residue forms are thus $\langle 1,-\alpha \beta\rangle $ and $\langle \alpha \beta,-\alpha,-\beta,\alpha \beta \rangle$ (the computation takes into account that in the residue field $k(\alpha,\beta)$, the element $t$ is replaced by $-b^{-1} a$, and then whenever $t$ appears as a coefficient, it can be multiplied by a square and we get $-\alpha \beta$).
    The form $\langle 1,-\alpha \beta \rangle$ is anisotropic (= not isotropic). The form $\langle \alpha \beta,-\alpha,-\beta,\alpha \beta \rangle=\alpha \beta \langle 1,1,-\alpha,-\beta \rangle$ is anisotropic too, because $-1$ is not a square in $k$, a contradiction.
\end{proof}

\begin{prop} \label{central_pol_ideal}
Given a central polynomial $p=\sum_{r,s} b_{r,s} x^r y^s$ in $D[x,y]$ for any division ring $D$ (i.e., $p \in F[x,y]$ where $F=Z(D)$), all the polynomials in the ideal generated by $p$ in $D[x,y]$ vanish at any root $(\alpha,\beta)$ of $p$.
\end{prop}

\begin{proof}
    The polynomial $p$ is central, so its left ideal is two-sided. Take a polynomial $h=qp$ in this ideal, $q=\sum_{k,\ell} a_{k,\ell} x^k y^\ell$. Then $h=\sum_{k,\ell,r,s} a_{k,\ell} b_{r,s} x^{k+r} y^{\ell+s}$, and so $h(\alpha,\beta)=\sum_{k,\ell,r,s} a_{k,\ell} b_{r,s} \alpha^{k+r} \beta^{\ell+s}$. But $b_{r,s}$ are central, and so $h(\alpha,\beta)=\sum_{k,\ell} a_{k,\ell} \alpha^k (\sum_{r,s} b_{r,s} \alpha^{r} \beta^{s})\beta^{\ell}=\sum_{k,\ell} a_{k,\ell} \alpha^k p(\alpha,\beta)\beta^{\ell}=0$.
\end{proof}

Let $p$ be the polynomial from Proposition \ref{counter_ex}, and let $I$ be the ideal generated by $p$. Clearly, $p$ vanishes at $(i,j)$. By proposition \ref{central_pol_ideal}, every polynomial in $I$ vanishes at $(i,j)$ as well. The polynomial $f=x^2$ does not vanish at this point, but it trivially vanishes on $\V_c(I)$ (because $\V_c(I)=\emptyset$). This provides a negative answer to the question in \cite{AP24}.

The field $F$ in the above example is of transcendence degree $3$ over the base field $k$. We do not know whether one could construct a similar example where the transcendence degree of the underlying field over its prime subfield is less than 3. More generally, a natural problem is to characterize the division rings over which the analogue of Theorem \ref{thm:central_zeros} holds. These questions remain open.

\section*{Acknowledgements}
The authors thank Jean-Pierre Tignol for helping formulate the argument in the proof of  Proposition \ref{counter_ex}.

\end{document}